\documentclass[12pt]{amsart}

\usepackage[T1]{fontenc}
\usepackage{amsmath,amssymb}
\usepackage{amscd}
\usepackage{pb-diagram}
\newtheorem{prop}{Theorem}[section]
\newtheorem{lemma}[prop]{Lemma}
\newtheorem{corollary}[prop]{Corollary}
\newtheorem{definition}[prop]{Definition}

\newcommand{\clc}{\cdot\ldots\cdot}
\newcommand{\olo}{\otimes\ldots\otimes}

\newcommand{\wlw}{\wedge\ldots\wedge}

\newcommand{\nn}{\mathbb{N}}
\newcommand{\zz}{\mathbb{Z}}

\newcommand{\cc}{\mathbb{C}}
\newcommand{\C}[1]{\mathcal{#1}}

\newcommand{\T}[1]{\textrm{#1}}
\newcommand{\E}[1]{\emph{#1}}

\newcommand{\F}[1]{\mathbf{#1}}
\newcommand{\fork}[2]{\left\{ \begin{array}{#1} #2 \end{array} \right.} 
\newcommand{\sep}[1]{& \T{#1 } &}
\newcommand{\arr}[2]{\begin{array}{#1} #2 \end{array}}

\newcommand{\q}{\qquad}

\newcommand{\diff}[2]{\frac{\partial{#1}}{\partial{#2}}}

\newcommand{\ccyc}[2]{C^{\lambda}_{#2}(#1)}
\newcommand{\hcyc}[2]{HC_{#2}(#1)}

\newcommand{\kcyc}[2]{Z^{\lambda}_{#2}(#1)}

\newcommand{\cech}[1]{\check{#1}}

\begin{document}
 \title{Comparison of secondary invariants of algebraic $K$-theory}

\author{J. Kaad}
\thanks{email:
    \texttt{kaad@math.ku.dk}}
\maketitle
\vspace{-10pt}
\centerline{ Department of Mathematical Sciences, University of Copenhagen}
\centerline{ Universitetsparken 5, DK-2100 Copenhagen, Denmark}

\bigskip
\vspace{8pt} 

\centerline{\textbf{Abstract}} 
In this paper we prove that the multiplicative character of A. Connes and
M. Karoubi and the determinant invariant of L. G. Brown, J. W. Helton and
R. E. Howe agree up to a canonical homomorphism.   

\vspace{-8pt}
\tableofcontents
\newpage
\section{Introduction}
In their paper, \cite{ConKar}, A. Connes and M. Karoubi define a multiplicative
character on algebraic $K$-theory 
\[
\C M_F : K_{2p}(\C A ) \to \cc / (2\pi i)^p \zz 
\]
for each odd unital $2p$-summable Fredholm module $(F,H)$ over the unital
$\cc$-algebra $\C A$. The construction uses the relative $K$-groups of a
unital Banach algebra and the relative Chern character with values in
continuous cyclic homology.

In a different direction, L.G. Brown has defined a determinant invariant 
\[
d_{(X,\iota)} : K_2(\C B) \to \cc^* 
\]
for each exact sequence of $\cc$-algebras
\[
\begin{CD}
X : 0 @>>> \C L^1(H) @>i>> \C E @>\pi>> \C B @>>> 0 
\end{CD}
\]
equipped with a unital algebra homomorphism $\iota : \C E \to \C L(H)$ such
that
\begin{equation}\label{eq:inccond}
\begin{array}{ccc}
(\iota \circ i)(T) = T
& \T{for all }
& T \in \C L^1(H)
\end{array}
\end{equation}
The construction uses the Fredholm determinant homomorphism 
\[
\arr{ccc}{
\T{det} : \C G \to \cc^* & \q & e^T \mapsto e^{\T{Tr}(T)}
}
\]
Here $\C G \subseteq GL(\C L(H))$ denotes the operators of determinant
class. See \cite{LBrown, LBrown2}. 

In the case where $\C B$ is commutative, the determinant invariant is related
to the work of J. W. Helton and R. E. Howe on traces of commutators,
\cite{HH}, via the identity
\[
\arr{ccc}{
d_{(X,\iota)} \{\pi(e^S),\pi(e^T)\} = e^{\T{Tr}[S,T]} & \q & S,T \in \C E
}
\]
Here $\{\pi(e^S),\pi(e^T)\} \in K_2(\C B)$ denotes the Steinberg
symbol. Notice also the appearance of the determinant invariant in the papers
of R. W. Carey and J. D. Pincus. See \cite{Carey1, Carey2}.   

The purpose of this paper is to show that the multiplicative character and the
determinant invariant coincide up to a canonical homomorphism. In particular,
to each odd unital $2$-summable Fredholm module $(F,H)$ over a unital
$\cc$-algebra $\C A$ we associate an extension of a certain $\cc$-algebra $\C
B$ by the operators of trace class
\begin{equation}\label{eq:exsecconspec}
\begin{CD}
X_F : 0 @>>> \C L^1(H) @>i>> \C E @>\pi>> \C B @>>> 0
\end{CD}
\end{equation}
an injective homomorphism $\iota : \C E \to \C L(H)$ satisfying
\eqref{eq:inccond}, and a surjective homomorphism 
\[
R_F : \C A \to \C B
\]
The main result of the paper can then be expressed as the commutativity of the
diagram  
\begin{equation}\label{eq:mainrecomm}
 \dgARROWLENGTH=0.6\dgARROWLENGTH
\begin{diagram}
\node{K_2(\C A)} \arrow{e,t}{(R_F)_*} \arrow{s,l}{M_F}
     \node{K_2(\C B)} \arrow{s,r}{d_{(X_F,\iota)}} \\
\node{\cc/(2\pi i)\zz} \arrow{e,b}{\exp}
     \node{\cc^*}
\end{diagram}
\end{equation}
It follows that the multiplicative character admits calculation in the
commutative case, and that it factorizes through the second algebraic $K$-group
of the Calcin algebra $\C C^1(H) = \C L(H)/\C L^1(H)$. 

Using the theory of central extensions we can state the main result in a
different fashion. In \cite[Paragraph $5$]{ConKar} it is shown that the
universal multiplicative character
\[
\C M_U : K_2(\C M^1) \cong H_2(E(\C M^1)) \to \cc/(2\pi i)\zz \cong \cc^*
\]
is induced by a central extension 
\begin{equation}\label{eq:conext}
\begin{CD}
\varphi: 1 @>>> \cc^* @>>> \Gamma @>>> E(\C M^1) @>>> 1
\end{CD}
\end{equation}
On the other hand, the universal determinant invariant
\[
d : K_2(E(\C C^1(H))) \cong H_2(E(\C C^1(H))) \to \cc^*
\]
could be understood as the homomorphism induced by the central extension 
\begin{equation}\label{eq:detext}
\begin{CD}
\psi : 1 @>>> \cc^* @>>> \widetilde{\Gamma} @>E(q)>> E(\C C^1(H)) @>>> 1
\end{CD}
\end{equation}
where $\widetilde{\Gamma} = GL(\C L(H))/\C G_0$ with $\C G_0$ defined as the
normal subgroup consisting of operators of determinant one. The main theorem
then translates into the following statement: The elements in group cohomology
with coefficients in $\cc^*$ determined by the central extensions
\eqref{eq:conext} and \eqref{eq:detext} coincide up to a homomorphism
\[
\arr{ccc}{
E(R)^*(\psi) = \varphi 
& \q & 
E(R)^* : H^2(E(\C C^1(H)),\cc^*) \to H^2(E(\C M^1),\cc^*)  
}
\]
We thus obtain refinements of results given in \cite{Segal}. Here we have
applied the universal coefficient Theorem together with the perfectness of the
elementary matrices.

The present paper is organized as follows.

In the first two sections we recall the construction of the determinant
invariant and the multiplicative character. Our definition of the determinant
invariant differs slightly from the one given in \cite{LBrown, LBrown2}. We
use a description of the relative homology groups of the surjective homomorphism 
\[
E(q) : E(\C L(H)) \to E(\C C^1(H))
\]
in terms of an algebraic analogue of the $\check{\T{C}}$ech complex of an open cover
of a topological space. See \cite{Inassa, Levine}. A combinatorial argument can be given,
proving that the two definitions yield the same map. The alternative approach
seems to be essential for the working of our main proof. 

In the last section we give a proof of the main result. There are two main
ingredients. The first ingredient is a factorization result for the
homomorphism 
\[
\arr{ccc}{
\tau_1 : HC_1(\C M^1) \to \cc & \q & (x_0,x_1) \mapsto
-\frac{1}{4}\T{Tr}(F_U[F_U,x_0][F_U,x_1]) 
}
\]
through a relative continuous cyclic homology group. This result is
essentially contained in the proof of \cite[Theorem $5.6$]{ConKar}. The
second ingredient is an application of the (well known) description of the
Fredholm determinant using a logarithm. To be precise, for each smooth map 
\[ 
\arr{ccc}{
\sigma : [0,1] \to \C G & \q & \sigma(0) = 1
}
\]
we can calculate the Fredholm determinant of the endpoint $\sigma(1) \in \C
G$, using the integral 
\[
\T{Tr}(\int_0^1 \frac{d\sigma}{dt} \cdot \sigma^{-1}dt) \in \cc
\]
We use this observation in a relative setting, invoking the algebraic
$\check{\T{C}}$ech complex a second time. These considerations lead directly
to a proof of the main result. 

We would like to end this introduction by mentioning that even though the
determinant invariant has no obvious generalization to the higher $K$-groups, 
we are still able to calculate the multiplicative character in this
setting, obtaining higher dimensional analogues of the identity 
\[
(\exp \circ \C M_F)\{e^T, e^S\} = \exp(\T{Tr}[PTP,PSP])
\]
This will be the subject of a forthcoming paper. 

{\bf Acknowledgements:} I would like to thank Ryszard Nest for his continuous support and
many helpful comments.  

\section{The determinant invariant of an extension by $\C
  L^1(H)$}\label{detkdef} 
The main subject of this section is the definition of the determinant
invariant. Our definition relies on considerations in relative homology, which
will find application throughout the paper. We will therefore start by a
formulation of these ideas in a general context before passing on to the
definition of the determinant invariant. 

\subsection{The $\check{\T{C}}$ech complex of a surjective simplicial map}\label{ccech}
Let $X$ and $Y$ be simplicial sets and suppose that we have a
\emph{surjective} simplicial map $f : X \to Y$. We let $\T{Ker}_*(f)$ denote
the kernel chain complex of the associated chain map $f : \zz[X_*] \to
\zz[Y_*]$. Furthermore we let $\cech{C}(f)$ denote the bicomplex with 
\[
\cech{C}_{nm}(f) = \zz[\underbrace{X_n \times_{Y_n} \ldots \times_{Y_n} X_n}_{m+2}]
\]
The generators are thus $(m+2)$-tuples of elements in $X_n$,
$(x_1,\ldots,x_{m+2})$, which coincide in $Y_n$, $f(x_1)=\ldots =
f(x_{m+2})$. The differentials are given by  
\[
\arr{ccc}{
d : \cech{C}_{nm}(f) \to \cech{C}_{(n-1)m}(f) 
& \q & d(x_1,\ldots,x_{m+2}) = \sum_{i=0}^n(-1)^i(d_i(x_1),\ldots,d_i(x_{m+2}))
}
\]
and 
\[
\arr{ccc}{
\delta : \cech{C}_{nm}(f) \to \cech{C}_{n(m-1)}(f)
& \q & \delta(x_1,\ldots,x_{m+2}) 
= \sum_{j=1}^{m+2}(-1)^{j+1}(x_1,\ldots, \cech{x}_j,\ldots,x_{m+2})
}
\]
Here the $\cech{\cdot}$ signifies that the term has to be omitted. We then
have a long exact sequence of abelian groups 
\[
 \dgARROWLENGTH=0.5\dgARROWLENGTH
\begin{diagram}
\node{\zz[X_n]} 
\node{\zz[Y_n]} \arrow{w,t}{f} 
\node{\cech{C}_{n0}(f)} \arrow{w,t}{\varepsilon}
\node{\cech{C}_{n1}(f)} \arrow{w,t}{\delta} 
\node{\cech{C}_{n2}(f)} \arrow{w,t}{\delta}
\node{\ldots} \arrow{w,t}{\delta}
\end{diagram}
\]
for each $n \in \nn_0$. Here the chain map 
\[
\varepsilon : \cech{C}_{*0}(f) \to \zz[Y_*]
\]
is given by $\varepsilon : (x_1,x_2) \mapsto x_1 - x_2$. In particular,
letting $\T{Coker}_*(\delta_f)$ denote the cokernel chain complex of the chain
map 
\[
\delta : \cech{C}_{*1}(f) \to \cech{C}_{*0}(f)
\]
we have the result 

\begin{prop}\label{isomrelg}\cite{Levine}
For each $n \in \nn_0$ we have the isomorphisms 
\[
H_n(f) \cong H_n(\E{Ker}(f)) \cong H_n(\E{Coker}(\delta_f)) \cong H_n(\cech{C}(f))
\]
Here $H_n(f)$ denotes the relative homology of the simplicial map $f : X \to Y$
\end{prop}





\subsection{The determinant invariant}
Let $\C C^1(H) = \C L(H)/\C L^1(H)$ denote the Calkin algebra. We then have
the quotient map $q : \C L(H) \to \C C^1(H)$ and an induced surjective group
homomorphism $E(q) : E(\C L(H)) \to E(\C C^1(H))$. Here $E(A)$ denotes the
elementary matrices of a unital ring $A$. 

We define the determinant homomorphism on the abelian group $\cech{C}_{10}(E(q))$ by 
\[
\arr{ccc}{
\T{det} : \zz[E(\C L(H)) \times_{E(\C C^1(H))} E(\C L(H))] \to \cc^* 
& \q & (g_1,g_2) \mapsto \T{det}(g_1g_2^{-1})
}
\]
Note that $g_1g_2^{-1}$ is of determinant class since $E(q)(g_1)=E(q)(g_2)$. 

\begin{lemma}
The determinant 
\[
\E{det} :  \zz[E(\C L(H)) \times_{E(\C C^1(H))} E(\C L(H))] \to \cc^*
\]
induces a map on homology 
\[
\E{det} : H_1(\E{Coker}(\delta_{E(q)})) \to \cc^*
\]
\end{lemma}
\begin{proof}
This is a matter of checking the equality
\[
\arr{ccc}{
\T{det}(g_1g_2^{-1})\cdot \T{det}(g_2g_3^{-1}) = \T{det}(g_1g_3^{-1}) 
\sep{for} E(q)(g_1) = E(q)(g_2) = E(q)(g_3)
}
\]
and the equality 
\[
\T{det}(g_1g_2^{-1})\T{det}(h_1h_2^{-1}) = \T{det}(g_1h_1h_2^{-1}g_2^{-1})
\]
for $E(q)(g_1)=E(q)(g_2)$ and $E(q)(h_1) = E(q)(h_2)$. But they follow easily
from well known properties of the Fredholm determinant, \cite{Simon, Wasser}.  
\end{proof}

By Theorem \ref{isomrelg}, the surjectivity of $E(q) : E(\C L(H)) \to E(\C
C^1(H))$ implies that 
\[
\varepsilon : H_1(\T{Coker}(\delta_{E(q)})) \to H_1(E(q))
\]
is an isomorphism. Furthermore, by \cite[Proposition $1.2.1$]{Loday} we can
identify the abelian groups $K_2(\C C^1(H))$ and $H_2(E(\C C^1(H)))$. We can
thus define the universal determinant invariant as the composition
\[
d = \T{det} \circ \partial : K_2(\C C^1(H)) \cong H_2(E(\C C^1(H))) \to \cc^*
\]
Here $\partial : H_2(E(\C C^1(H))) \to H_1(E(q))$ is the boundary map
associated with the group homomorphism $E(q) : E(\C L(H)) \to E(\C C^1(H))$.

Now, suppose that we have an exact sequence of $\cc$-algebras 
\[
\begin{CD}
X : 0 @>>> \C L^1(H) @>i>> \C E @>\pi>> \C B @>>> 1
\end{CD}
\]
equipped with a unital algebra homomorphism $\iota : \C E \to \C L(H)$ such
that
\[
\arr{ccc}{
(\iota \circ i)(T) = T 
\sep{for all}
T \in \C L^1(H)
}
\]
We then have an induced homomorphism $\iota : \C B \to \C C^1(H)$ which by
functoriality of algebraic $K$-theory yields a homomorphism 
\[
\iota_* : K_2(\C B) \to K_2(\C C^1(H))
\]
We define the determinant invariant of the pair $(X,\iota)$ as the composition
\[
d_{(X,\iota)} = d \circ \iota_* : K_2(\C B) \to \cc^*
\]

\section{The multiplicative character of a finitely summable Fredholm
  module}\label{multchar} 
In this section we recall the construction of the multiplicative character
associated with an odd, finitely summable Fredholm module. To this end we will
describe the relative Chern character and the relative $K$-groups in detail.
The references for this section are \cite{ConKar, Karoubi} and
\cite{Tillmann}. Note that the cyclic homology and Hochschild homology groups
encountered in this section, and throughout the paper are the non-Hausdorff
versions of \emph{continuous} cyclic homology and \emph{continuous} Hochschild
homology groups.

\subsection{The relative Chern character}\label{relchernsec}
Let $A$ be a unital Banach algebra. We let  
\[
\Delta^n = \{t \in [0,1]^n\, |\, \sum_{i=1}^nt_i \leq 1 \}
\]
denote the standard $n$-simplex. The vertices will be denoted by
$\F{0},\ldots,\F{n} \in \Delta^n$. 

Let $R(A)$ denote the simplicial set with 
\[
R(A)_n = \{\sigma \in GL(C^\infty(\Delta^n,A)) \, | \, \sigma(\F{0})=1 \}
\]
in degree $n \in \nn_0$ and with face operators and degeneracy operators
defined by 
\[
\begin{split}
& d_i(\sigma)(t_1,\ldots,t_{n-1}) = 
\fork{ccc}{
\sigma(1-\sum_{j=1}^{n-1}t_j,t_1,\ldots,t_{n-1})\cdot \sigma(\F{1})^{-1} 
\sep{for} 
j = 0 \\
\sigma(t_1,\ldots,t_{i-1},0,t_i,\ldots,t_{n-1})
\sep{for} j \in \{1,\ldots,n\} \\
} \\
& s_j(\sigma)(t_1,\ldots,t_{n+1}) = 
\fork{ccc}{
\sigma(t_2,\ldots,t_{n+1}) 
\sep{for} j = 0 \\
\sigma(t_1,\ldots,t_{i-1},t_i + t_{i+1},\ldots,t_{n+1}) 
\sep{for} j \in \{1,\ldots,n\} 
}
\end{split}
\]
Remark the extra factor $\sigma(\F{1})^{-1}$ in the expression for $d_0$. 

Let $|R(A)|$ denote the geometric realization of the simplicial set
$R(A)$. The fundamental group is then given by 
\[
\pi_1(|R(A)|) = R(A)_1/\sim
\] 
where $\sim$ denotes the equivalence relation given by smooth homotopies with
fixed endpoints. See \cite{May}. Since the commutator subgroup is perfect and
normal, we can apply the plus construction of D. Quillen, obtaining the pointed
$CW$-complex $|R(A)|^+$. See \cite{Quillen}.

\begin{definition}\cite{Karoubi}
By \emph{the relative} $K$-\emph{theory} of $A$ we will understand the
homotopy groups of $|R(A)|^+$ thus by definition 
\[
K_n^{\E{rel}}(A) = \pi_n(|R(A)|^+)
\]
for each $n \in \nn_0$. 
\end{definition}

By \cite{Karoubi} the relative $K$-groups fit in a long exact sequence 
\begin{equation}\label{eq:xseqk}
 \dgARROWLENGTH=0.5\dgARROWLENGTH
\begin{diagram}
\node{\ldots } \arrow{e,t}{i}
     \node{K_{n+1}^{\T{top}}(A)} \arrow{e,t}{v}
          \node{K_n^{\T{rel}}(A)} \arrow{e,t}{\theta}
               \node{K_n(A)} \arrow{s,r}{i} \\
               \node{\ldots } 
          \node{K_{n-1}(A)} \arrow{w,b}{i} 
     \node{K_{n-1}^{\T{rel}}(A)} \arrow{w,b}{\theta}     
\node{K_n^{\T{top}}(A)} \arrow{w,b}{v}
\end{diagram}
\end{equation}
which terminates at $K_1^{\T{top}}(A)$. Here $\theta : K_n^{\T{rel}}(A) \to
K_n(A)$ is induced by the simplicial map  
\[
\arr{ccc}{
\theta : R_n(A) \to GL(A)^n  
& \q & \theta(\sigma) =
(\sigma(\F{0})\sigma(\F{1})^{-1},\ldots,\sigma(\F{n-1})\sigma(\F{n})^{-1}) 
}
\]
Furthermore, in \cite{ConKar} it is shown that the long exact sequence
\eqref{eq:xseqk} is related to the continuous $SBI$-sequence in homology by
means of Chern characters 
\[
 \dgARROWLENGTH=0.5\dgARROWLENGTH
\begin{diagram}
\node{\q \ldots } \arrow{e,b}{i}    
     \node{K_{n+1}^{\T{top}}(A)} \arrow{e,t}{v} \arrow{s,l}{\T{ch}^{\T{top}}_{n+1}} 
          \node{K_n^{\T{rel}}(A)} \arrow{e,t}{\theta}
          \arrow{s,l}{\frac{(-1)^n}{(n-1)!}\T{ch}_n^{\T{rel}}}
               \node{K_n(A)} \arrow{e,t}{i} \arrow{s,l}{\frac{1}{n!} D_n}
                    \node{K_n^{\T{top}}(A)} \arrow{e,t}{v} \arrow{s,l}{\T{ch}^{\T{top}}_n}
                         \node{\ldots \q}   \\
\node{\q \ldots} \arrow{e,b}{I}
     \node{HC_{n+1}(A)} \arrow{e,b}{S} 
          \node{HC_{n-1}(A)} \arrow{e,b}{\frac{1}{n}B}
               \node{HH_n(A)} \arrow{e,b}{I}
                    \node{HC_n(A)} \arrow{e,b}{S}
                         \node{\ldots \q} 
\end{diagram}
\]
We will give a precise description of the relative Chern character 
\[
\T{ch}_n^{\T{rel}} : K_n^{\T{rel}}(A) \to HC_{n-1}(A)
\]
as the composition of three homomorphisms. The
first one is the Hurewicz homomorphism  
\[
h_n : K_n^{\T{rel}}(A) \to H_n(|R(A)|^+) \cong H_n(R(A))
\]
The second one is the logarithm 
\[
L : H_n(R(A)) \to
\lim_{m\to \infty} HC_{n-1}(M_m(A))
\]
which is induced by 
\begin{equation}\label{eq:logarithm}
L : \sigma \mapsto \frac{1}{n}
                \sum_{s\in \Sigma_n}\T{sgn}(s) 
                     \int_{\Delta^n} \diff{\sigma}{t_{s(1)}}\cdot \sigma^{-1} \olo
                                    \diff{\sigma}{t_{s(n)}}\cdot \sigma^{-1} 
                                         dt_1\wlw dt_n
\end{equation}
for each smooth map $\sigma : \Delta^n \to GL_m(A)$. The last one is the
generalized trace on continuous cyclic homology 
\[
\T{TR} : \lim_{m\to \infty} HC_{n-1}(M_m(A)) \to HC_{n-1}(A)
\]

\begin{definition}\cite{ConKar,Karoubi}
By \emph{the relative Chern character}
\[
\E{ch}_n^{\E{rel}} : K_n^{\E{rel}}(A) \to HC_{n-1}(A)
\]
we will understand the composition 
\begin{equation}\label{eq:chrel}
\E{ch}_n^{\E{rel}} = \E{TR} \circ L \circ h_n
\end{equation}
of the Hurewicz homomorphism, the logarithm and the generalized trace. 
\end{definition}





\subsection{The multiplicative character}
Let $H$ be a separable Hilbert space of infinite dimension. For each $p \in
\nn$ let $\C M^{2p-1}$ denote the unital $\cc$-subalgebra of $\C L(H\oplus H)$
consisting of operators of the form
\[
\left(
\begin{array}{cc}
x_{11} & x_{12} \\
x_{21} & x_{22} 
\end{array}
\right) \in \C L(H\oplus H)
\] 
with $x_{12},x_{21}\in \C L^{2p}(H)$ in the $2p^{\T{th}}$ Schatten ideal. The
$\cc$-algebra $\C M^{2p-1}$ becomes a unital Banach algebra when equipped with
the norm
\[
\arr{ccc}{
\|x\| = \|x\|_\infty + \|\,[F_U,x]\,\|_{2p} 
& \q & 
F_U = \left( \arr{cc}{
1 & 0 \\
0 & -1 
}
\right)
}
\]
Here the norms $\|\cdot \|_\infty$ and $\|\cdot\|_{2p}$ are the operator norm
and the norm on the $2p^{\T{th}}$ Schatten ideal respectively. For details on
the Schatten ideals we refer to \cite{Simon}. 

The continuous linear map  
\[
\arr{ccc}{
\tau_{2p-1} : \ccyc{\C M^{2p-1}}{2p-1} \to \cc 
& &
(x_0,\ldots,x_{2p-1}) \mapsto 
- \frac{1}{2^{2p}(p-1)!} \T{Tr}(F_U[F_U,x_0]\clc [F_U,x_{2p-1}])
}
\]
determines a continuous cyclic cocycle and consequently a homomorphism
\[
\tau_{2p-1} : HC_{2p-1}(\C M^{2p-1}) \to \cc
\]
For details we refer to \cite{Connesgeom}. Pre-composition with the relative
Chern character thus yields a homomorphism 
\[
\tau_{2p-1} \circ \T{ch}_{2p}^{\T{rel}} : K_{2p}^{\T{rel}}(\C M^{2p-1}) \to \cc
\]

In \cite{ConKar} it is shown that $K_{2p}(\C M^{2p-1})$ is the cokernel of 
\[
v : K_{2p+1}^{\T{top}}(\C M^{2p-1}) \to K_{2p}^{\T{rel}}(\C M^{2p-1})
\]
from the exact sequence \eqref{eq:xseqk}. Furthermore it is shown that the
image of the homomorphism 
\[
\tau_{2p-1} \circ \T{ch}_{2p}^{\T{rel}} \circ v : K_{2p+1}^{\T{top}}(\C
M^{2p-1}) \to \cc
\]
equals the additive subgroup $(2\pi i)^{p}\zz $ of $\cc$. By consequence
we get a homomorphism 
\[
\C M_U : K_{2p}(\C M^{2p-1})\cong \T{Coker}(v) \to \cc/(2\pi i)^p \zz
\]
This is the odd universal multiplicative character.

Now, to each odd unital $2p$-summable Fredholm $(F,H)$ over a unital
$\cc$-algebra $\C A$ we can associate a unital algebra homomorphism $ \rho_F :
\C A \to \C M^{2p-1} $ which by functoriality of algebraic $K$-theory yields a
homomorphism
\[
(\rho_F)_* : K_{2p}(\C A) \to K_{2p}(\C M^{2p-1})
\]
The multiplicative character of the Fredholm module $(F,H)$ over $\C A$ is
then defined as the composition  
\[
M_F = M \circ (\rho_F)_* : K_{2p}(\C A) \to \cc/(2\pi i)^{p}\zz
\]





\section{Comparing the determinant invariant and the multiplicative character}    

\subsection{Factorization through relative continuous cyclic
  homology}\label{factor} 
In this section we show that the homomorphism 
\[
\tau_1 : HC_1(\C M^1) \to \cc
\]
factorizes through a relative continuous cyclic homology group, associated with
an extension of $\C M^1$ by the operators of trace class. This result can be
found in condensed form in the proof of \cite[Theorem $5.6$]{ConKar}. 

Let $H$ be a separable Hilbert space and let $\C M^1$ be the Banach algebra
considered in Section \ref{multchar}. Let $P$ be the projection 
\[
P=\left(
\begin{matrix}
1 & 0 \\
0 & 0
\end{matrix}
\right) \in \C L(H\oplus H)
\]
and let $q : \C L(H)\to \C L(H)/\C L^1(H)= \C C^1(H)$ be the quotient map. We
then have a unital algebra homomorphism
\[
\arr{ccc}{
R : \C M^1\to \C C^1(H) & \q & R : x \mapsto q(PxP)
}
\]

Let $\C T^1\subseteq \C L(H)\times \C M^1$ be the unital $\cc$-subalgebra such
that $(S,x)\in \C T^1$ precisely when $S-PxP\in \C L^1(H)$. The diagram
\begin{equation}\label{eq:commpisec}
 \dgARROWLENGTH=0.5\dgARROWLENGTH
\begin{diagram}
\node{\C T^1} \arrow{e,t}{\pi_2} \arrow{s,l}{\pi_1} 
\node{\C M^1} \arrow{s,r}{R} \\
\node{\C L(H)} \arrow{e,b}{q} \node{\C L(H)/\C L^1(H)} 
\end{diagram}
\end{equation}
is thus a commutative diagram of unital $\cc$-algebras. Here $\pi_1$ and
$\pi_2$ are the projections given by $\pi_1(S,x)=S$ and $\pi_2(S,x)=x$.

The $\cc$-algebra $\C T^1$ becomes a unital Banach algebra when equipped with
the norm
\[
\|(S,x)\|=\|PxP-S\|_1 + \|x\|_\infty + \|[2P-1,x]\|_2 
\]
It fits in the short exact sequence of Banach algebras 
\begin{equation}\label{eq:banexa}
\begin{CD}
0 @>>> \C L^1(H) @>i>> \C T^1 @>\pi_2>> \C M^1 @>>> 0
\end{CD} 
\end{equation}
which has a continuous linear section 
\[
\arr{ccc}{
s : \C M^1 \to \C T^1 & \q & x \mapsto (PxP,x)
}
\]

The existence of the continuous linear right inverse to $\pi_2$ yields an
isomorphism 
\[
HC_*(\pi_2) \cong HC_*(\T{Ker}(\pi_2))
\]
between the relative continuous cyclic homology of $\pi_2$ and the homology of
the kernel complex associated with the chain map 
\[
(\pi_2)_* : C_*^\lambda(\C T^1) \to C_*^\lambda(\C M^1)
\]
In particular, each element in $HC_0(\pi_2)$ can be represented by an operator
of trace class 
\[
\begin{array}{ccc}
(S,0) \in \C T^1
& \quad 
& S \in \C L^1(H)
\end{array}
\]

\begin{lemma}\label{reltrace}
The trace 
\[
\arr{ccc}{
\E{Tr} : \E{Ker}(\pi_2) \to \cc & \q & (S,0) \mapsto \E{Tr}(S)
}
\]
induces a map on relative continuous cyclic homology 
\[
T : HC_0(\pi_2) \to \cc
\]
\end{lemma}
\begin{proof}
Let $\alpha : \ccyc{\C M^1}{1} \to \C L(H)$ denote the map induced by 
\[
\alpha : (x_0,x_1) \mapsto [Px_0P,Px_1P]
\]
Let $\beta : \ccyc{\C T^1}{1} \to \C L^1(H)$ be the map induced by 
\[
\beta : ((S_0,x_0),(S_1,x_1)) \mapsto [S_0,S_1] - [Px_0P,Px_1P]  
\]
We then have 
\[
\T{Tr} \circ \beta = 0
\]
Furthermore we can express the composition of the Hochschild boundary 
\[
b : \ccyc{\C T^1}{1} \to \ccyc{\C T^1}{0} = \C T^1
\]
with the projection $\pi_1 : \C T^1 \to \C L(H)$ as the sum 
\[
\pi_1 \circ b = \beta + \alpha \circ (\pi_2)_*
\]
It follows that the map 
\[
\T{Tr} \circ b : \T{Ker}_1(\pi_2) \to \cc
\]
vanishes, which in turn implies the desired result. 
\end{proof}

We can now prove that $\tau_1 : HC_1(\C M^1) \to \cc$ factorizes through the
relative continuous cyclic homology group $\hcyc{\pi_2}{0}$. 

\begin{prop}\label{cocycdiff}
The character 
\[
\tau_1 : HC_1(\C M^1) \to \cc
\]
factorizes as 
\[
\tau_1 = - T \circ \partial 
\]
Here $\partial : HC_1(\C M^1) \to HC_0(\pi_2)$ is the boundary map associated
with the homomorphism $\pi_2 : \C T^1 \to \C M^1$. 
\end{prop}
\begin{proof}
We start by noting that the composition 
\[
- \T{Tr} \circ ( b \circ s_* -s_* \circ b) : C_1^{\lambda}(\C M^1) \to \cc
\]
coincides with the map
\[
\tau_1 : C_1^{\lambda}(\C M^1) \to \cc
\]
Here $s_* : \ccyc{\C M^1}{1} \to \ccyc{\C T^1}{1}$ denotes the map induced by
the continuous linear right inverse $s : \C M^1 \to \C T^1$ of $\pi_2 : \C T^1
\to \C M^1$. Since the boundary map
\[
\partial : HC_1(\C M^1) \to HC_0(\pi_2) \cong HC_0(\T{Ker}(\pi_2)) 
\]  
is induced by the map 
\[
b \circ s_* : \kcyc{\C M^1}{1} \to \T{Ker}(\pi_2)
\]
we get the desired result.
\end{proof}  

As an immediate consequence of Theorem \ref{cocycdiff} we get the following 
factorization result for the composition of the relative Chern character and the
homomorphism $\tau_1 : HC_1(\C M^1) \to \cc$.  

\begin{corollary}\label{partch}
The composition  
\[
\tau_1 \circ \E{ch}_2^{\E{rel}} : K_2^{\E{rel}}(\C M^1) \to \cc 
\]
factorizes as 
\[
\tau_1 \circ \E{ch}_2^{\E{rel}} 
= - T \circ \E{TR} \circ L \circ \partial \circ h_2
\]
Here the maps 
\[
\E{TR} : \lim_{m \to \infty} HC_1(M_m(\pi_2)) \to HC_1(\pi_2)
\]
and 
\[
L : H_2(R(\pi_2)) \to \lim_{m \to \infty} HC_1(M_m(\pi_2))
\]
are relative versions of the generalized trace and the logarithm. The map 
\[
\partial : H_2(R(\C T_1)) \to H_1(R(\pi_2))
\]
is the boundary map associated with the simplicial map $R(\pi_2) : R(\C T^1)
\to R(\C M^1)$.
\end{corollary}





\subsection{Determinants and the relative logarithm}\label{detrellog} 
In this part of the paper we start relating the multiplicative character with
the Fredholm determinant. The main result of this section is thus Lemma
\ref{fredholm}, where we show that the composition of the relative logarithm 
\[
L : H_1(R(\pi_2)) \to \lim_{m\to
  \infty} HC_0(M_m(\pi_2))
\]
with the exponential of the trace  
\[
\exp \circ (- T \circ \T{TR}) : \lim_{m\to \infty} HC_0(M_m(\pi_2)) \to \cc^*
\]
is a Fredholm determinant. 

\begin{definition}\label{smoothlone}
Let $\C G$ denote the operators of determinant class. By a smooth map $\sigma
: \Delta^n \to \C G$ we will understand an element in the group 
\[
\sigma \in GL\big(C^\infty(\Delta^n, \C L(H))\big)
\]
such that   
\[
\sigma - 1 \in  M_\infty\big(C^\infty(\Delta^n, \C L^1(H))\big) 
\] 
\end{definition}

Define the map 
\[
\tilde{L} : \zz [R(\C T^1)_1 \times_{R(\C M^1)_1} R(\C T^1)_1]
=\check{\T{C}}_{10}(R(\pi_2)) \to \cc     
\]
by the assignment 
\[
\tilde{L} : (\sigma_1,\sigma_2) \mapsto 
              -\T{Tr} \int_0^1\frac{d (\sigma_1 \sigma_2^{-1})}{dt}
                                 \cdot \sigma_2 \sigma_1^{-1} dt
\] 
Here the trace 
\[
\T{Tr} : M_\infty(\C L^1(H)) \to \cc
\]
is induced by 
\[
\arr{ccc}{
\T{Tr} : x \mapsto \sum_{i=1}^m \T{Tr}(x_{ii})
\sep{for all} 
x \in M_m\big(\C L^1(H)\big)
}
\]

\begin{lemma}\label{tildegamma}
The map  
\[
\tilde{L} : \zz[R(\C T^1)_1 \times_{R(\C M^1)_1} R(\C T^1)_1] \to \cc
\]
passes to a map on homology 
\[
\tilde{L} : H_1(\E{Coker}(\delta_{R(\pi_2)})) \to \cc
\]
making the diagram 
\begin{equation}\label{eq:tilgamgam}
 \dgARROWLENGTH=0.5\dgARROWLENGTH
\begin{diagram}
\node{H_1(R(\pi_2))} 
\arrow{s,l}{-T\circ \E{TR}\circ L} 
     \node{H_1(\E{Coker}(\delta_{R(\pi_2)}))} 
     \arrow{w,t}{\varepsilon} 
     \arrow{sw,r}{\tilde{L}} \\
\node{\cc} 
     \node{} 
\end{diagram}
\end{equation}
commute.  
\end{lemma}
\begin{proof}
We show that 
\[
\tilde{L} : \zz[R(\C T^1)_1 \times_{R(\C M^1)_1} R(\C T^1)_1] \to \cc 
\]
agrees with the map 
\[
-T \circ \T{TR} \circ L \circ \varepsilon :
\zz[R(\C T^1)_1 \times_{R(\C M^1)_1} R(\C T^1)_1] \to \cc
\]
See Section \ref{ccech}. Here $\T{TR} \circ L : \T{Ker}_1(R(\pi_2))  \to
\T{Ker}(\pi_2)$ is induced by the restriction of $\T{TR} \circ L : \zz[R(\C
T^1)_1] \to \C T^1$ to the kernel of $R(\pi_2) : \zz[R(\C T^1)_1] \to
\zz[R(\C M^1)_1]$.

This essentially amounts to a proof of the equality 
\[
\tilde{L}(\sigma_1,\sigma_2) 
= \T{Tr} \big(
\int_0^1 \frac{d\sigma_2}{dt}\cdot \sigma_2^{-1} dt -
\int_0^1 \frac{d\sigma_1}{dt}\cdot \sigma_1^{-1} dt 
 \big)
\]
for all $(\sigma_1,\sigma_2) \in R(\C T^1)_1 \times_{R(\C M^1)_1} R(\C
T^1)_1$.

Now, since the trace 
\[
\T{Tr} : M_m(\C L^1(H)) \to \cc
\]
is continuous and linear we have 
\[
\begin{split}
\tilde{L}(\sigma_1,\sigma_2) 
& = -\T{Tr} \int_0^1\frac{d (\sigma_1 \sigma_2^{-1})}{dt}
\cdot \sigma_2 \sigma_1^{-1} dt \\
& = -\int_0^1 \T{Tr} \big(
\frac{d (\sigma_1 \sigma_2^{-1})}{dt}
\cdot \sigma_2 \sigma_1^{-1}\big) dt
\end{split}
\]
Using the Leibnitz rule for the differential operator $\frac{d}{dt}$ we get  
\[
\frac{d(\sigma_1 \sigma_2^{-1})}{dt} 
= \frac{d\sigma_1}{dt}\cdot \sigma_2^{-1} 
-\sigma_1 \sigma_2^{-1}\cdot  \frac{d\sigma_2}{dt} \cdot \sigma_2^{-1}  
\]
Defining $\alpha \in M_\infty\big(C^{\infty}(\Delta^1,\C L^1(H))\big)$ by
$\alpha + 1 = \sigma_2 \sigma_1^{-1}$ we have
\[
\begin{split}
& -\T{Tr} \big( \frac{d (\sigma_1 \sigma_2^{-1})}{dt}
                             \cdot \sigma_2 \sigma_1^{-1}\big) \\
&\qquad = -\T{Tr} \big( \frac{d\sigma_1}{dt} \cdot \sigma_1^{-1} 
-(\alpha+1)^{-1}\frac{d\sigma_2}{dt} \cdot \sigma_2^{-1}(\alpha+1) \big)  \\ 
&\qquad = \T{Tr} \big((\alpha+1)^{-1} \frac{d\sigma_2}{dt}\cdot \sigma_2^{-1}
\alpha\big)
+ \T{Tr} \big( (\alpha+1)^{-1} \frac{d\sigma_2}{dt} \cdot \sigma_2^{-1} 
-\frac{d\sigma_1}{dt} \cdot \sigma_1^{-1} \big) \\ 
&\qquad =  \T{Tr} \big(\frac{d\sigma_2}{dt} \cdot \sigma_2^{-1} -
\frac{d\sigma_1}{dt} \cdot \sigma_1^{-1} \big) 
\end{split}
\]
proving the desired result. 
\end{proof}

In the next Lemma we will show that the composition of the map 
\[
\tilde{L} : H_1(\T{Coker}(\delta_{R(\pi_2)})) \to \cc
\]
with the exponential 
\[
\exp : \cc \to \cc^*
\]
can be expressed using a Fredholm determinant. 

By \cite[Proposition $5.4$]{ConKar} the first algebraic $K$-group of the
Banach algebra $\C T^1$ vanishes. In particular, it follows from the exact
sequence \eqref{eq:banexa} that the groups $E(\C M^1)$ and $GL_0(\C M^1)$
coincide. Here $GL_0(\C M^1)$ denotes the connected component of the
identity. We therefore have a commutative diagram of simplicial maps 
\[
 \dgARROWLENGTH=0.6\dgARROWLENGTH
\begin{diagram}
\node{R_n(\C T^1)} \arrow{e,t}{\theta} \arrow{s,l}{R(\pi_2)}
     \node{E(\C T^1)^n} \arrow{s,r}{E(\pi_2)} \\
\node{R_n(\C M^1)} \arrow{e,b}{\theta}
     \node{E(\C M^1)^n}  
\end{diagram}
\]
and by consequence an induced map  
\[
\theta : H_1(\T{Coker}(\delta_{R(\pi_2)})) \to H_1(\T{Coker}(\delta_{E(\pi_2)}))
\]
Furthermore we let 
\[
D_* : H_1(\T{Coker}(\delta_{E(\pi_2)})) \to H_1(\T{Coker}(\delta_{E(q)}))
\]
denote the map induced by the commutative diagram 
\[
\dgARROWLENGTH=0.6\dgARROWLENGTH
\begin{diagram}
\node{E(\C T^1)} \arrow{e,t}{E(\pi_1)} \arrow{s,l}{E(\pi_2)}
     \node{E(\C L(H))} \arrow{s,r}{E(q)} \\
\node{E(\C M^1)} \arrow{e,b}{E(R)}
     \node{E(\C C^1(H))}  
\end{diagram}
\]
of group homomorphisms. 
 
\begin{lemma}\label{fredholm}
The  diagram 
\begin{equation}\label{eq:fredcomm}
 \dgARROWLENGTH=0.5\dgARROWLENGTH
\begin{diagram}
\node{H_1(\E{Coker}(\delta_{R(\pi_2)}))}
\arrow{s,r}{\tilde{L}} \arrow{e,t}{D_* \circ \theta} 
           \node{H_1(\E{Coker}(\delta_{E(q)}))}
           \arrow{s,r}{\E{det}}  \\
\node{\cc} \arrow{e,b}{\E{exp}} 
     \node{\cc^*}
\end{diagram}
\end{equation}
is commutative. 
\end{lemma}
\begin{proof}
We show that for each generator 
\[
(\sigma_1,\sigma_2) 
\in R(\C T^1)_1 \times_{R(\C M^1)_1} R(\C T^1)_1
\]
we have 
\[
(\exp \circ \tilde{L})(\sigma_1,\sigma_2) 
= \T{det} \big(E(\pi_1)(\sigma_1(\F 1)^{-1}) \cdot E(\pi_1)(\sigma_2(\F 1))\big)
\]
Indeed the smooth map $\sigma_1\sigma_2^{-1} : \Delta^1 \to GL(\C T^1)$ is of
the form $\sigma_1\sigma_2^{-1} = (\alpha,1)$ with $\alpha : \Delta^1 \to \C
G$ being a smooth map in the sense of Definition \ref{smoothlone}. It follows
that
\[
(\exp \circ \tilde{L})(\sigma_1,\sigma_2)
= \exp\big(- \T{Tr} \int_0^1 \frac{d \alpha}{ dt} \cdot \alpha^{-1}dt \big) 
\]  
But this is precisely the Fredholm determinant of the endpoint 
\[
\T{det}(\alpha(\F 1)^{-1}) 
= \T{det}\big( E(\pi_1)(\sigma_2(\F 1)) \cdot E(\pi_1)(\sigma_1(\F 1)^{-1})\big)
\]
proving the desired result. 
\end{proof}





\subsection{Proof of the main result}\label{proof}
In this section we combine the results of Lemma \ref{cocycdiff} and Lemma
\ref{fredholm}, in order to obtain a proof of our main theorem: the
determinant invariant equals the multiplicative character up to a canonical
homomorphism on algebraic $K$-theory. 

We start by making a couple of useful observations. Let 
\[
\theta : H_2(R(\C M^1)) \to H_2(E(\C M^1))
\]
denote the homomorphism induced by the simplicial map 
\[
\theta : R_n(\C M^1) \to E(\C M^1)^n
\]
The composition 
\[
\theta \circ h_2 : K_2^{\T{rel}}(\C M^1) \to H_2(E(\C M^1))
\]
then corresponds to the map 
\[
\theta : K_2^{\T{rel}}(\C M^1) \to K_2(\C M^1)
\]
under the identification $K_2(\C M^1) \cong H_2(E(\C M^1))$. 

Furthermore, the simplicial map 
\[
R(\pi_2) : R(\C T^1) \to R(\C M^1)
\]
is surjective in each degree. This follows from \cite[Proposition
$3.4.3$]{Black} by the surjectivity of the map 
\[
M_m(\pi_2) : M_m(C^\infty(\Delta^n, \C T^1) ) 
\to M_m(C^\infty(\Delta^n, \C M^1)) 
\]
and since each element $\sigma \in R(\C M^1)_n$ can be represented by an
element
\[
\sigma \in GL_0\big(M_m(C^\infty(\Delta^n,\C M^1))\big) 
\]
By Theorem \ref{isomrelg} the map 
\[
\varepsilon : H_*(\T{Coker}(\delta_{R(\pi_2)})) \to H_*(R(\pi_2)) 
\] 
thus becomes an isomorphism.  

We are now ready to prove the main result of the paper. 

\begin{prop}\label{mainthe}
The composition of the universal multiplicative character
\[
\C M_U : K_2(\C M^1) \to \cc/(2\pi i)\zz
\]
with the exponential function 
\[
\exp : \cc/(2\pi i)\zz \to \cc^*
\]
coincides with the composition of the homomorphism on algebraic $K$-theory 
\[
R_* : K_2(\C M^1) \to K_2(\C C^1(H))
\]
induced by the unital algebra homomorphism
\[
\begin{array}{ccc}
R : \C M^1 \to \C C^1(H) 
& \quad 
& x\mapsto q(PxP)
\end{array}
\]
and the universal determinant invariant
\[
d : K_2(\C C^1(H)) \to \cc^*
\]
That is 
\[
\exp \circ \C M_U = d \circ R_*
\]
\end{prop}
\begin{proof}
From the long exact sequence \eqref{eq:xseqk} we see that 
\[
\theta : K_2^{\T{rel}}(\C M^1) \to K_2(\C M^1)
\]
implements the isomorphism 
\[
\T{Coker}(v)\cong K_2(\C M^1)
\]
By construction of the multiplicative character it is therefore enough to
prove the equality 
\[
\exp \circ \tau_1 \circ \T{ch}_2^{\T{rel}} = d \circ R_* \circ \theta 
\]
See Section \ref{multchar}. We evolve on the left hand side. The
identifications 
\[
\arr{ccc}{
H_1(R(\pi_2)) \cong H_1\big(\T{Coker}(\delta_{R(\pi_2)})\big)  
\sep{and}
H_1(E(q)) \cong H_1\big(\T{Coker}(\delta_{E(q)})\big)
}
\]
as well as the identifications 
\[
\arr{ccc}{
K_2(\C M^1) \cong H_2(E(\C M^1)) \sep{and}
K_2(\C C^1(H)) \cong H_2(E(\C C^1(H)))
}
\]
will be suppressed. Furthermore we will make use of the results in Corollary
\ref{partch}, Lemma \ref{tildegamma} and Lemma \ref{fredholm} as well as of the
observations preceeding the statement of the Theorem. We get 
\[
\begin{split}
\exp \circ \tau_1 \circ \T{ch}_2^{\T{rel}} 
& = \exp \circ ( - T \circ \T{TR} \circ L \circ \partial \circ h_2) \\
& = \exp \circ \tilde{L} \circ \partial \circ h_2 \\
& = \T{det} \circ D_* \circ \theta \circ \partial \circ h_2 \\
& = \T{det} \circ \partial \circ E(R)_* \circ \theta \circ h_2 \\
& = \T{det} \circ \partial \circ R_* \circ \theta 
\end{split}
\]
But this is precisely the composition 
\[
d \circ R_* \circ \theta : K_2^{\T{rel}}(\C M^1) \to \cc^*
\]
proving the theorem. 
\end{proof}

The main theorem can be refined in the following way. Let 
\[
\rho_F : \C A \to \C M^1
\]
be the unital algebra homomorphism associated with an odd unital $2$-summable
Fredholm module $(F,H)$ over the $\cc$-algebra $\C A$. We then define a
$\cc$-subalgebra of $\C L(H)$ by
\[
\C E = \{P\rho_F(a)P + T \, |\, T \in \C L^1(H) \T{ and } a \in \C A\}
\] 
Letting $\C B$ denote the quotient of $\C E$ by $\C L^1(H)$ we get an exact
sequence 
\[
\begin{CD}
X_F : 0 @>>> \C L^1(H) @>i>> \C E @>\pi>> \C B @>>> 0 
\end{CD}
\]
We then have the algebra homomorphism $\iota : \C B \to \C C^1(H)$ which is
induced by the inclusion $\iota : \C E \to \C L(H)$ and the surjective algebra
homomorphism  
\[
\arr{ccc}{
R_F : \C A \to \C B 
& \q & 
a \mapsto q(P\rho_F(a)P)
}
\]
Here $q : \C L(H) \to \C C^1(H)$ denotes the quotient map.

\begin{corollary}
The diagram 
\[
\dgARROWLENGTH=0.5\dgARROWLENGTH 
\begin{diagram}
\node{K_2(\C A)}
\arrow{e,t}{(R_F)_*} \arrow{s,l}{\C M_F}
     \node{K_2(\C B)} 
     \arrow{s,r}{d_{(X_F,\iota)}} \\
\node{\cc/(2\pi i)\zz} 
\arrow{e,b}{\exp}
     \node{\cc^*}
\end{diagram}
\]
is commutative. 
\end{corollary}
\begin{proof}
The result follows from Theorem \ref{mainthe} and the identity 
\[
R \circ \rho_F = \iota \circ R_F : \C A \to \C C^1(H)
\]
using the definition of the invariants and the functoriality of algebraic
$K$-theory. 
\end{proof}






\begin{thebibliography}{99999}

\bibitem{Black} B. Blackadar, {\em K-Theory for Operator Algebras},  
Mathematical Sciences Research Institute Publications, Springer-Verlag New
York, $1986$.  

\bibitem{Brown} K. S. Brown, 
{\em Cohomology of Groups}, Graduate Texts in Math., vol. $87$, Springer-Verlag,
New York, Heidelberg, Berlin, $1982$. 

\bibitem{LBrown} L. G. Brown, 
{\em The Determinant Invariant for Operators with Trace Class Self
  Commutators}, Proc. Conf. on Operator Theory, Lecture Notes in Math.,
vol. $345$, Springer-Verlag, Berlin, Heidelberg, and New York, $1973$, $210$-$228$.  

\bibitem{LBrown2} L. G. Brown, 
{\em Operator Algebras and Algebraic K.theory}, Bull. Amer. Math. Soc., {\bf
  81}, $1973$, $1119$-$1121$.

\bibitem{Carey1} R. W. Carey, J. D. Pincus.  
{\em Perturbation Vectors}, Integral Equations Operator Theory, {\bf 35},
, $1999$, $271$-$365$.

\bibitem{Carey2} R. W. Carey, J. D. Pincus.
{\em Steinberg Symbols Modulo the Trace Class, Holonomy, and Limit Theorems
  for Toeplitz Determinants}, Transactions of the American Mathematical
Society, vol. $358$, $2004$, $509$-$551$. 

\bibitem{Cartan} H. Cartan and S. Eilenberg, 
{\em Homological Algebra}, Princeton Mathematical Series, vol. $19$,
Princeton Univ. Press, Princeton, $1956$.

\bibitem{Connesgeom} A. Connes, \emph{Noncommutative Geometry}, Academic Press,
  $1994$.

\bibitem{ConKar} A. Connes and M. Karoubi,
{\em Caractère multiplicatif d'un module de Fredholm}, K-theory {\bf 2},
$1988$, $431$-$463$.
 
\bibitem{Dupont} J. L. Dupont,
{\em Curvature and Characteristic Classes}, Lecture Notes in Mathematics,
vol. $640$, Springer Verlag, Berlin, Heidelberg, New York, $1978$.

\bibitem{Hatcher} A. Hatcher,
{\em Algebraic Topology}, Cambridge University Press, $2002$. 

\bibitem{HH} J. W. Helton and R. E. Howe,
{\em Integral Operators: Commutators, Traces, Index, and Homology},
Proc. Conf. on Operator Theory, Lecture Notes in Math., vol. $345$,
Springer-Verlag, Berlin, Heidelberg, New York, $1973$, $141$-$209$.   

\bibitem{Inassa} G. Donadze, N. Inassaridze, T. Porter,
{\em N-fold $\check{\T{C}}$ech Derived Functors and Generalized Hopf Type
  Formulas} $K$-theory preprint archives, {\bf 624}, $2003$. 

\bibitem{Ginot} G. Ginot, 
{\em Formules explicites pour le caractère de Chern en K-théorie algébrique},
Ann. Inst. Fourier, Grenoble {\bf 54}, $2004$, $2327$-$2355$.

\bibitem{Grothendieck} A. Grothendieck,
{\em Produits Tensoriels Topologiques}, Mem. AMS., No. $16$, $1955$.

\bibitem{Karoubi} M. Karoubi,
{\em Homologie Cyclique et K-théorie}, Astérisque No. $149$, $1987$. 

\bibitem{Levine} M. Levine,  
{\em Relative Milnor K-Theory}, K-theory {\bf 6}, $1992$, $113$-$175$.

\bibitem{Loday} J. -L. Loday,
{\em K-théorie algébrique et représentations de groupes}, 
Ann. Scient. Ecole Norm. Sup. {\bf 9}, $1976$, $309$-$377$.  

\bibitem{Loday2} J. -L. Loday,
{\em Cyclic Homology}, Grundlehren der mathematischen Wissenschaften,
vol. $301$, Springer-Verlag, New York, Berlin, Heidelberg, $1992$.

\bibitem{LodQuil} J. -L. Loday and D. Quillen, 
{\em Cyclic Homology and The Lie Algebra Homology of Matrices},
Comment. Math. Helv. {\bf 59}, $1984$, $565$-$591$.

\bibitem{macduff} D. Macduff and P. de la Harpe,
{\em Acyclic groups of automorphisms}, Comment. Math. Helv. {\bf 58}, $1983$,
$48$-$71$.  

\bibitem{May} J. P. May, 
{\em Simplicial Objects in Algebraic Topology}, Chicago Lectures in
Mathematics, $1993$. 

\bibitem{Pedersen} G. K. Pedersen, {\em $C^*$-Algebras and Their Automorphism
    Groups},   
Academic Press, London, New York, San Francisco, 1979.

\bibitem{Quillen} D. Quillen, 
{\em Higher Algebraic K-theory I}, Algebraic K-theory I (Battelle, $1972$)
(H. Bass, ed.), Lecture Notes in Math., vol. $341$, Springer-Verlag, New York,
Berlin, Heidelberg, $1973$, $85$-$147$.

\bibitem{Rosenberg} J. Rosenberg, 
{\em Algebraic K-theory and Its Applications}, Graduate Texts in Math.,
vol. $147$, Springer-Verlag, New York, Berlin, Heidelberg, $1994$. 

\bibitem{Rosenberg2} J. Rosenberg, 
{\em Comparison Between Algebraic and Topological K-Theory for Banach
  Algebras and $C^*$-Algebras}, Handbook of Algebraic $K$-theory,
Springer, $2004$, $843$-$874$.

\bibitem{Segal} A. Presley and G. Segal,
{\em Loop Groups}, Clarendon Press, Oxford, $1986$. 

\bibitem{Simon} B. Simon
{\em Trace Ideals and Their Applications}, Cambridge University Press, $1979$.  

\bibitem{Tillmann} U. Tillmann, 
{\em Factorization of the (Relative) Chern Character through Lie Algebra
  Homology}, K-theory {\bf 6}, $1992$, $457$-$463$.

\bibitem{Wagoner} J. B. Wagoner, 
{\em Delooping Classifying Spaces in Algebraic $K$-theory}, Topology,
Vol. $11$, $1972$, $349$-$370$. 

\bibitem{Wasser} A. Wassermann, 
{\em Analysis of Operators}, Michaelmas, $2006$.

\bibitem{Whitehead} G. W. Whitehead, 
{\em Elements of Homotopy Theory}, Graduate Texts in Math., vol. $61$,
Springer-Verlag, New York, Heidelberg, Berlin, $1978$.

\end{thebibliography}
\end{document}